\documentclass[12pt,oneside]{amsart}
\usepackage{amsaddr}

\usepackage[a4paper]{geometry}

\usepackage{amsmath,amssymb,amsthm,mathrsfs,latexsym,mathtools,mathdots,booktabs,enumerate,tikz,tikz-cd,bm,url,cleveref}

\theoremstyle{plain}
\newtheorem{theorem}{Theorem}[section]
\newtheorem{lemma}[theorem]{Lemma}

\newtheorem{proposition}[theorem]{Proposition}

\theoremstyle{definition}
\newtheorem{definition}[theorem]{Definition}
\newtheorem{example}[theorem]{Example}

\theoremstyle{remark}
\newtheorem{remark}[theorem]{Remark}

\newcommand{\defin}[1]{\textcolor{blue}{\emph{#1}}}

\usepackage{lipsum}
\usepackage{todonotes}
\usepackage{mathtools, enumerate}
\usepackage[centertableaux]{ytableau}

\DeclareMathOperator{\sgn}{sgn}

\DeclareMathOperator{\height}{ht}
\DeclareMathOperator{\maj}{maj}

\DeclareMathOperator{\DES}{DES}

\newcommand{\SSYT}{\mathrm{SSYT}}
\newcommand{\SYT}{\mathrm{SYT}}
\newcommand{\SYTT}{\mathrm{SYT\text{-}tuples}}
\newcommand{\SSYTT}{\mathrm{SSYT\text{-}tuples}}
\newcommand{\BST}{\mathrm{BST}}
\DeclareMathOperator{\stat}{des^+}
\DeclareMathOperator{\idx}{idx}
\newcommand{\tm}{\text{-}} 

\newcommand{\xvec}{{\boldsymbol x}}

\newcommand{\emphDes}[1]{{\underline{#1}}}

\numberwithin{equation}{section}

\usepackage[backend=bibtex, maxnames=4]{biblatex}
\title[Descents for BST]{A refinement of the Murnaghan--Nakayama rule by descents for border strip tableaux}

\author{Stephan Pfannerer}
\thanks{Stephan Pfannerer was partially supported by  the  Austrian  Science  Fund(FWF) P29275 and is a recipient of a DOC Fellowship of the Austrian Academy of Sciences.}
\address{Institut f\"ur Diskrete Mathematik und Geometrie, TU Wien, Austria}
\email{stephan.pfannerer@tuwien.ac.at}

\begin{document}

\begin{abstract}
Lusztig's fake degree is the generating polynomial for the major index of
standard Young tableaux of a given shape. Results of Springer and James \&
Kerber imply that, mysteriously, its evaluation at a $k$-th primitive root
of unity yields the number of border strip tableaux with all strips of
size $k$, up to sign. This is essentially the special case of the Murnaghan--Nakayama rule for evaluating an irreducible character of the symmetric group at a rectangular partition.

We refine this result to standard Young tableaux and border strip tableaux
with a given number of descents. To do so, we introduce a new statistic for
border strip tableaux, extending the classical definition of descents in
standard Young tableaux. Curiously, it turns out that our new statistic
is very closely related to a descent set for tuples of standard Young
tableaux appearing in the quasisymmetric expansion of LLT polynomials
given by Haglund, Haiman and Loehr.

\end{abstract}

\maketitle

\section{Introduction}

Let $\SYT(\lambda)$ denote the set of all \defin{standard Young tableaux} of shape $\lambda$ and size $n$. 
An entry $i$ of a standard Young tableau $T$ is a \defin{descent} of T, if $i+1$ appears in a strictly lower row in $T$ in English notation. Let $\DES(T)$ denote the set of descents of $T$ and denote with $\maj(T)$, the \defin{major index} of $T$, the sum of all descents of $T$.

Our main result, \cref{thm:mainThm}, provides a natural combinatorial
interpretation of

\[
f^{\lambda}(q,t)\coloneqq \sum_{T\in \SYT(\lambda)} q^{\maj(T)}t^{|\DES(T)|}
\]
when $q$ is a root of unity.
This refines the classical interpretation of the evaluation of Lusztig's fake degree polynomial $f^\lambda(q) \coloneqq f^\lambda(q,1)$, as we show below. The bivariate generating function itself was already considered by R. Stanley \cite[Proposition 8.13.]{Stanley1971} in the more general setting of $(P, \omega)$ partitions.

By a result of T. Springer \cite[Proposition 4.5]{Springer1974}, $f^\lambda(q)$ coincides with the restriction of the irreducible $\mathfrak{S}_n$-character $\chi^\lambda$ to the cyclic subgroup generated by the long cycle, represented as the group of complex roots of unity. More precisely, for $k\mid n$, let $\xi$ be a primitive $k$-th root of unity and let $\rho=(k^{n/k})$ be a rectangular partition, then we have $f^\lambda(\xi) = \chi^\lambda(\rho)$.
In general, a practical way to compute the character value $\chi^{\lambda}(\rho)$  for an arbitrary partition $\rho$ is the Murnaghan--Nakayama rule \cite[Theorem 7.17.3]{Stanley1999}.

For a filling $B$ of a Young diagram $\lambda$ with weakly increasing rows and columns, let $B^i$ be the collection of cells in $B$ containing $i$. We say that $B$ is a \defin{border strip tableau} of type $\rho= (\rho_1,\dots,\rho_\ell)\vdash n$ if, for all $1\le i \le \ell$, the cells $B^i$ form a connected skew shape of size $\rho_i$ that does not contain a $2\times 2$ rectangle.
In this case we call $B^i$ a \defin{border strip} of size $\rho_i$ and define the \defin{height} $\height(B^i)$ to be the number of rows it spans minus $1$. Furthermore, we define the height of $B$ to be the sum of the heights of its strips. See \cref{fig:BST} for an example of a border strip tableau of height $7$.

\begin{figure}
\[ B = 
\begin{tikzpicture}[x=1.2em, y=1.2em,baseline={([yshift=-1ex]current bounding box.center)}]
\draw[line width=0.8] (0,0) -- (6,0) -- ++(0,-1) -- ++(-1,0) -- ++(0,-1) -- ++(-1,0) -- ++(0,-1) -- ++(-2,0) -- ++(0,-1) -- ++(-0,0) -- ++(0,-1) -- ++(-0,0) -- ++(0,-1) -- ++(-2,0) -- cycle;
\draw[line width=0.8] (0,0) -- (4,0) -- ++(0,-1) -- ++(-0,0) -- ++(0,-1) -- ++(-0,0) -- ++(0,-1) -- ++(-2,0) -- ++(0,-1) -- ++(-0,0) -- ++(0,-1) -- ++(-0,0) -- ++(0,-1) -- ++(-2,0);
\draw[line width=0.8] (0,0) -- (4,0) -- ++(0,-1) -- ++(-0,0) -- ++(0,-1) -- ++(-0,0) -- ++(0,-1) -- ++(-2,0) -- ++(0,-1) -- ++(-1,0) -- ++(0,-1) -- ++(-1,0);
\draw[line width=0.8] (0,0) -- (4,0) -- ++(0,-1) -- ++(-0,0) -- ++(0,-1) -- ++(-0,0) -- ++(0,-1) -- ++(-4,0);
\draw[line width=0.8] (0,0) -- (3,0) -- ++(0,-1) -- ++(-0,0) -- ++(0,-1) -- ++(-0,0) -- ++(0,-1) -- ++(-3,0);
\draw[line width=0.8] (0,0) -- (3,0) -- ++(0,-1) -- ++(-0,0) -- ++(0,-1) -- ++(-3,0);
\draw[line width=0.8] (0,0) -- (2,0) -- ++(0,-1) -- ++(-1,0) -- ++(0,-1) -- ++(-1,0);
\draw[line width=0.8] (0,0) -- (0,0);

\draw (4.5,-1.5) node{7};\draw (4.5,-0.5) node{7};\draw (5.5,-0.5) node{7};
\draw (0.5,-5.5) node{6};\draw (1.5,-5.5) node{6};\draw (1.5,-4.5) node{6};
\draw (0.5,-4.5) node{5};\draw (0.5,-3.5) node{5};\draw (1.5,-3.5) node{5};
\draw (3.5,-2.5) node{4};\draw (3.5,-1.5) node{4};\draw (3.5,-0.5) node{4};
\draw (0.5,-2.5) node{3};\draw (1.5,-2.5) node{3};\draw (2.5,-2.5) node{3};
\draw (1.5,-1.5) node{2};\draw (2.5,-1.5) node{2};\draw (2.5,-0.5) node{2};
\draw (0.5,-1.5) node{1};\draw (0.5,-0.5) node{1};\draw (1.5,-0.5) node{1};
\end{tikzpicture}
\]
\caption{A border strip tableau of height $7$.}\label{fig:BST}
\end{figure}

The Murnaghan--Nakayama rule states that
\[
\chi^{\lambda}(\rho) = \sum_{B}(-1)^{\height(B)},
\]
where the sum is taken over all border strip tableaux of shape $\lambda$ and type $\rho$. In the special case where $\rho$ is a rectangular partition $(k^{n/k})$ and all strips in $B$ have the same size $k$, this rule is cancellation free, due to a theorem by G. James and A. Kerber \cite[Theorem 2.7.27]{James1984}.
This means that the parity of $\height(B)$ only depends on $\lambda$ and the strip size $k$.
Thus we obtain 
\begin{equation}\label{eq:fakedegreeEvaluation}
f^\lambda(\xi) = \epsilon_{\lambda, k}\cdot |\BST(\lambda, k)|,
\end{equation}
where $\BST(\lambda, k)$ is the set of all border strip tableaux of shape $\lambda$ and type $(k^{n/k})$, and $\epsilon_{\lambda, k} = (-1)^{\height(B)}$ for any border strip tableau $B\in \BST(\lambda,k)$.

Our main result, \cref{thm:mainThm}, refines this special case of the Murnaghan--Nakayama rule as follows.
Provided that the set $\BST(\lambda, k)$ is not empty, it turns out that $f^\lambda(\xi,t)$ is, up to sign, a generating function for a very natural statistic on this set. That is,
\[
f^{\lambda}(\xi,t) = \epsilon_{\lambda,k} \cdot \sum_{B\in \BST(\lambda,k)}t^{\stat(B)}.
\]
The statistic $\stat$ (see \cref{def:stat}) extends the classical definition of descents for standard Young tableaux.

The motivation to study $f^\lambda(\xi,t)$ comes from the desire to refine recent results by P. Alexandersson, S. Pfannerer, M. Rubey, and J. Uhlin \cite{alexandersson_pfannerer_rubey_uhlin_2021}. They used \eqref{eq:fakedegreeEvaluation} to show that $(\chi^\lambda)^2$ carries the action of a permutation representation of the cyclic group of order $n$. Equivalently, there exists a cyclic group action $\tau$ of order $n$ such that the triple $(\SYT(\lambda)\times \SYT(\lambda), \tau, (f^\lambda)^2)$ exhibits the cyclic sieving phenomenon introduced by V. Reiner, D. Stanton and D. White \cite{Reiner2004}. Note that the action $\tau$ remains unknown. Refining their work may eventually lead to an explicit description of $\tau$. See B. Sagan's survey article \cite{Sagan2011} for more background on the cyclic sieving phenomenon.

The structure of the paper is as follows: In \cref{sec:definitions} we introduce relevant definitions, fix notation and state our main theorem. The rest of the sections is dedicated for the proof, which is split up into three important steps. In \cref{sec:partitionQuotient} we introduce the Littlewood quotient map and in \cref{sec:schur} we relate our problem to Schur functions. Finally, in \cref{sec:bijection} we conclude the main result with a bijection.

\section{Definitions and main theorem}\label{sec:definitions}
We begin by introducing relevant definitions and notation; for more details we refer to the books by I. Macdonald \cite{Macdonald1995} and R. Stanley \cite{Stanley1999}.
A \defin{partition} $\lambda=(\lambda_1,\dots,\lambda_\ell)$ of $n$, written as $\lambda \vdash n$, is a weakly decreasing sequence of positive integers that sum up to $n\eqqcolon |\lambda|$.
The \defin{Young diagram} of shape $\lambda$ in English notation is the  collection of $n$ cells, arranged in $\ell$ left-justified rows of lengths $\lambda_1,\dots,\lambda_\ell$.
The rows of a Young diagram are indexed from top to bottom starting with one, and the columns are indexed from left to right starting with one.
To each cell $x$ we associate its \defin{content}, $c(x)$, which is its column index minus its row index. Finally, we associate to each cell $x$ its \defin{hook value}, $h(x)$, which is the number of cells to the right of $x$ plus the number of cells below $x$ plus one.

\begin{figure}
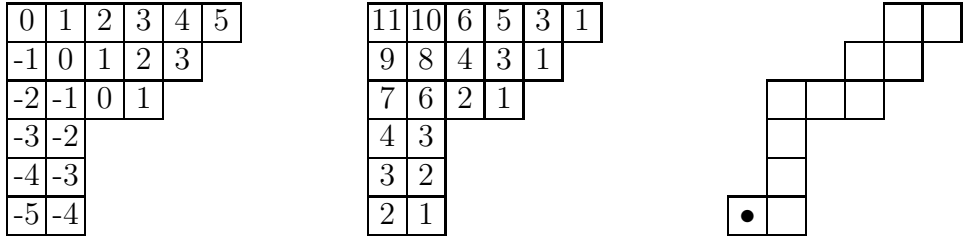

\ytableausetup{boxsize=1.2em}
\[
\ytableaushort{012345,{\tm1}0123,{\tm2}{\tm1}01,{\tm3}{\tm2},{\tm4}{\tm3},{\tm5}{\tm4}}\qquad\qquad
\ytableaushort{{11}{10}6531,98431,7621,43,32,21}\qquad\qquad \ytableaushort{\none\none\none\none{}{},\none\none\none{}{},\none{}{}{},\none{},\none{},{\bullet}{}}
\]
\caption{The Young diagram of the partition $654222$, where each cell is filled with its content, the same diagram where each cell is filled with its hook value and a border strip, where its tail is decorated with a bullet.}\label{fig:partition}
\end{figure}

Let $(\lambda,\mu)$ be a pair of partitions such that the Young diagram $\mu$ is completely contained in the Young diagram $\lambda$. The cells that are in $\lambda$ but not in $\mu$ form the \defin{skew shape} $\lambda/\mu$.

A \defin{border strip} (or ribbon or rim hook) is a connected skew shape that does not contain a $2\times 2$ square.
The \defin{tail} of a border strip is its unique cell with smallest content.
The \defin{height} $\height(S)$ of a border strip $S$ is the number of rows it spans minus one.

\begin{definition}
A \defin{border strip tableau} of shape $\lambda\vdash n$ with strip size $k\mid n$ is a Young diagram filled with the integers $\{1,\ldots,n/k\}$ such that
\begin{itemize}
\item the values in each row from left to right and each column from top to bottom are weakly increasing and
\item the cells containing the value $i$ form a border strip of size $k$ for all $1\le i\le n/k$.
\end{itemize}

The set of all such tableaux is denoted with $\BST(\lambda,k)$.
The \defin{height} $\height(B)$ of a border strip tableau $B$ is the sum of the heights of its border strips.
\end{definition}

\begin{example}
The Young diagram of $654222$ is given in \cref{fig:partition} on the left. In the same Figure on the right we see the skew shape $654222/43111$. It is a border strip of size $11$ and height~$5$.
 The tableau $B$ in \cref{fig:BST} is a border strip tableau in $\BST(654222,3)$ of height~$7$.
\end{example}

\begin{remark}
A border strip tableau with all strips having size $1$ corresponds to a \defin{standard Young tableau}. 
We also write $\SYT(\lambda)$ for $\BST(\lambda,1)$.
\end{remark}
One may also think of a border strip tableau $B\in \BST(\lambda,k)$ as a flag of partitions
$\emptyset = \nu_0 \subset \nu_1 \subset \dots \subset \nu_{n/k} = \lambda$ such that
$B^i\coloneqq\nu_{i}/\nu_{i-1}$ is a border strip of size $k$ for all $1\le i \le n/k$.
The cells in $B^i$ are precisely the cells in $B$ with label $i$.

More generally, we say that a partition $\nu$ is obtained from $\lambda$ \defin{by removing} a border strip of size $k$, if $\lambda/\nu$ is a border strip of size $k$.
Successively removing border strips of size $k$ from $\lambda$ as long as possible gives a partition $\nu_0$ which turns out to be independent from the order in which the strips are removed. Hence, $\nu_0$ is well defined and called the \defin{$k$-core} of $\lambda$. See also G. James and A. Kerber \cite[Chapter 2.7]{James1984} for further details.

\begin{proposition}
$\BST(\lambda,k)$ is not empty, if and only if $\lambda$ has empty $k$-core.
\end{proposition}

From now on, we remove all labels in the graphical representation of a border strip tableau that are not in the tail of a strip.
We now define the descent set of a border strip tableau, extending the classical definition for standard Young tableaux.
\begin{definition}
Let $x_i$ be the unique cell with smallest content in a border strip tableau $B$ that contains $i$.
We call $i$ a \defin{descent} of $B$ if $x_{i+1}$ appears in a strictly lower row in $B$ than $x_i$. $\DES(B)$ denotes the set of all descents of $B$.
\end{definition}
For example, let $B$ be the border strip tableau in \cref{fig:BST}. Following our new convention, $B$ is depicted again in \cref{fig:BSTdescents}. Hence $\DES(B)=\{2,4,5\}$.

\begin{figure}[h]
\centering
\[B = 
\begin{tikzpicture}[x=1em, y=1em,baseline={([yshift=-1ex]current bounding box.center)}]
\draw[line width=0.8] (0,0) -- (6,0) -- ++(0,-1) -- ++(-1,0) -- ++(0,-1) -- ++(-1,0) -- ++(0,-1) -- ++(-2,0) -- ++(0,-1) -- ++(-0,0) -- ++(0,-1) -- ++(-0,0) -- ++(0,-1) -- ++(-2,0) -- cycle;
\draw[line width=0.8] (0,0) -- (4,0) -- ++(0,-1) -- ++(-0,0) -- ++(0,-1) -- ++(-0,0) -- ++(0,-1) -- ++(-2,0) -- ++(0,-1) -- ++(-0,0) -- ++(0,-1) -- ++(-0,0) -- ++(0,-1) -- ++(-2,0);
\draw[line width=0.8] (0,0) -- (4,0) -- ++(0,-1) -- ++(-0,0) -- ++(0,-1) -- ++(-0,0) -- ++(0,-1) -- ++(-2,0) -- ++(0,-1) -- ++(-1,0) -- ++(0,-1) -- ++(-1,0);
\draw[line width=0.8] (0,0) -- (4,0) -- ++(0,-1) -- ++(-0,0) -- ++(0,-1) -- ++(-0,0) -- ++(0,-1) -- ++(-4,0);
\draw[line width=0.8] (0,0) -- (3,0) -- ++(0,-1) -- ++(-0,0) -- ++(0,-1) -- ++(-0,0) -- ++(0,-1) -- ++(-3,0);
\draw[line width=0.8] (0,0) -- (3,0) -- ++(0,-1) -- ++(-0,0) -- ++(0,-1) -- ++(-3,0);
\draw[line width=0.8] (0,0) -- (2,0) -- ++(0,-1) -- ++(-1,0) -- ++(0,-1) -- ++(-1,0);
\draw[line width=0.8] (0,0) -- (0,0);

\draw (4.5,-1.5) node{$7$};
\draw (0.5,-5.5) node{$6$};
\draw (0.5,-4.5) node{$\emphDes5$};
\draw (3.5,-2.5) node{$\emphDes4$};
\draw (0.5,-2.5) node{$3$};
\draw (1.5,-1.5) node{$\emphDes2$};
\draw (0.5,-1.5) node{$1$};
\end{tikzpicture}\]
\caption{The border strip tableau $B$ with the labels in the tails of the strips. The descents of $B$ are highlighted. We have $\DES(B)=\{2,4,5\}$.}
\label{fig:BSTdescents}
\end{figure}
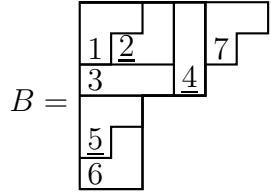

\begin{definition}
The \defin{major index} $\maj(T)$ of a standard Young diagram $T$ is the sum of its descents. We define the generating function
\[
f^{\lambda}(q,t)\coloneqq \sum_{T\in \SYT(\lambda)} q^{\maj(T)}t^{|\DES(T)|}.
\]
\end{definition}
Note that $f^{\lambda}(q,1)$ is also known as Lusztig's fake degree polynomial.
Also note that $f^{\lambda}(1,t)$ is the generating function for the number of descents on $\SYT(\lambda)$. 

Coming to our main definition, we introduce a new statistic on $\BST(\lambda, k)$.

\begin{definition}\label{def:stat}
Let $B\in \BST(\lambda, k)$, let $B^1$ be the strip in $B$ containing $1$ and define
\[
\stat(B) \coloneqq k\cdot |\DES(B)| + \height(B^1).
\]
\end{definition}

Observe that for $T\in \SYT(\lambda)=\BST(\lambda,1)$ this statistic equals the number of descents, that is $\stat(T)=|\DES(T)|$. For example, the tableau $B$ in \cref{fig:BST} has $\stat(B)=10$.

We can now state our main theorem.
\begin{theorem}\label{thm:mainThm}
Let $\lambda$ be a partition of $n$ with empty $k$-core and let $\xi$ be a primitive $k$-th root of unity. Then, for some $\epsilon_{\lambda,k}\in\{\pm1\}$,
\begin{equation}\label{eq:mainThm}
f^{\lambda}(\xi,t) = \epsilon_{\lambda,k} \cdot \sum_{B\in \BST(\lambda,k)}t^{\stat(B)}.
\end{equation}

\end{theorem}
Furthermore, the sign $\epsilon_{\lambda,k}$ can be made explicit. 
\begin{proposition}\label{thm:sign}Provided that $\BST(\lambda, k)$ is non-empty, the map
\begin{align*}
\sgn:\BST(\lambda,k) &\to \{\pm1\}\\
B&\mapsto (-1)^{\height(B)}
\end{align*}
is constant, 
and $\epsilon_{\lambda,k}= \sgn(B_0)$, where $B_0$ is any border strip tableau in $\BST(\lambda,k)$.
\end{proposition}
\begin{remark}
Let $\chi^{\lambda}$ be the irreducible character of the symmetric group corresponding to $\lambda$ and let $\rho=(k^{n/k})$ be the rectangular partition of $n$ with all parts equal to $k$. 
Then $\epsilon_{\lambda,k}$ is the sign of the character value $\chi^\lambda(\rho)$. 
\end{remark}

\begin{example} To illustrate \cref{thm:mainThm}  let us consider the partition $\lambda=222$.
There are five standard Young tableaux, as depicted below. For each tableau we highlight the descents and calculate its weight in $f^{222}(q,t)$.
\ytableausetup{boxsize=1.2em}
\[
\begin{array}{ccccccccc}
\ytableaushort{{\emphDes1}{\emphDes4},{\emphDes2}{\emphDes5},36} && \ytableaushort{{\emphDes1}{\emphDes3},2{\emphDes5},46} && \ytableaushort{1{\emphDes2},{\emphDes3}{\emphDes5},46} && \ytableaushort{{\emphDes1}{\emphDes3},2{\emphDes4},56} && \ytableaushort{1{\emphDes2},3{\emphDes4},56}\vspace{0.5em}\\
q^{12}t^4 &+& q^{9}t^3 &+& q^{10}t^3 &+& q^{8}t^3 &+& q^{6}t^2
\end{array}
\]
We obtain
\[
f^{222}(q,t) = q^{12}t^4 + (q^{10} + q^{9} + q^{8})t^3 + q^{6}t^2.
\]
Substituting primitive roots of unity for $q$ we obtain the following polynomials
\[\begin{tabular}{c|c|c|c|c}
primitive root& $1^{st}=1$ & $2^{nd}=-1$ & $3^{rd}$ & $6^{th}$\\\hline
$f^{222}(\cdot,t)$  & {$t^{4} + 3 t^{3} + t^{2}$}  & {$t^{4} + t^{3} + t^{2}$} & {$t^{4} + t^{2}$} & {$\underbrace{t^{4} -2 t^{3} + t^{2}}_{\text{$6$-core not empty}}$}.
\end{tabular}
\]
For $q=1$ we get the generating function for the number of descents on $\SYT(\lambda)=\BST(\lambda,1)$. For $q$ being a second or third root of unity we obtain the generating functions for $\stat$ on $\BST(\lambda,2)$ and $\BST(\lambda,3)$ respectively:
\[
\BST(\lambda,2):
\begin{array}{ccc}
\begin{tikzpicture}[x=1.2em, y=1.2em]
\draw (0,0) -- (2,0) -- ++(0,-1) -- ++(-0,0) -- ++(0,-1) -- ++(-0,0) -- ++(0,-1) -- ++(-2,0) -- cycle;
\draw (0,0) -- (2,0) -- ++(0,-1) -- ++(-0,0) -- ++(0,-1) -- ++(-2,0) -- cycle;
\draw (0,0) -- (2,0) -- ++(0,-1) -- ++(-2,0) -- cycle;
\draw (0,0) -- (0,0) -- cycle;
\draw (0.5,-2.5) node{$3$};
\draw (0.5,-1.5) node{$\emphDes2$};
\draw (0.5,-0.5) node{$\emphDes1$};
\end{tikzpicture}&
\begin{tikzpicture}[x=1.2em, y=1.2em]
\draw (0,0) -- (2,0) -- ++(0,-1) -- ++(-0,0) -- ++(0,-1) -- ++(-0,0) -- ++(0,-1) -- ++(-2,0) -- cycle;
\draw (0,0) -- (2,0) -- ++(0,-1) -- ++(-1,0) -- ++(0,-1) -- ++(-0,0) -- ++(0,-1) -- ++(-1,0) -- cycle;
\draw (0,0) -- (2,0) -- ++(0,-1) -- ++(-2,0) -- cycle;
\draw (0,0) -- (0,0) -- cycle;
\draw (1.5,-2.5) node{$3$};
\draw (0.5,-2.5) node{$2$};
\draw (0.5,-0.5) node{$\emphDes1$};
\end{tikzpicture}&
\begin{tikzpicture}[x=1.2em, y=1.2em]
\draw (0,0) -- (2,0) -- ++(0,-1) -- ++(-0,0) -- ++(0,-1) -- ++(-0,0) -- ++(0,-1) -- ++(-2,0) -- cycle;
\draw (0,0) -- (2,0) -- ++(0,-1) -- ++(-0,0) -- ++(0,-1) -- ++(-2,0) -- cycle;
\draw (0,0) -- (1,0) -- ++(0,-1) -- ++(-0,0) -- ++(0,-1) -- ++(-1,0) -- cycle;
\draw (0,0) -- (0,0) -- cycle;
\draw (0.5,-2.5) node{$3$};
\draw (1.5,-1.5) node{$\emphDes2$};
\draw (0.5,-1.5) node{$1$};
\end{tikzpicture}\\
t^{2\cdot 2+0} & t^{2\cdot 1+0} & t^{2\cdot 1+1}
\end{array}
\qquad\qquad
\BST(\lambda, 3):
\begin{array}{cc}
\begin{tikzpicture}[x=1.2em, y=1.2em]
\draw (0,0) -- (2,0) -- ++(0,-1) -- ++(-0,0) -- ++(0,-1) -- ++(-0,0) -- ++(0,-1) -- ++(-2,0) -- cycle;
\draw (0,0) -- (1,0) -- ++(0,-1) -- ++(-0,0) -- ++(0,-1) -- ++(-0,0) -- ++(0,-1) -- ++(-1,0) -- cycle;
\draw (0,0) -- (0,0) -- cycle;
\draw (1.5,-2.5) node{$2$};
\draw (0.5,-2.5) node{$1$};
\end{tikzpicture}&
\begin{tikzpicture}[x=1.2em, y=1.2em]
\draw (0,0) -- (2,0) -- ++(0,-1) -- ++(-0,0) -- ++(0,-1) -- ++(-0,0) -- ++(0,-1) -- ++(-2,0) -- cycle;
\draw (0,0) -- (2,0) -- ++(0,-1) -- ++(-1,0) -- ++(0,-1) -- ++(-1,0) -- cycle;
\draw (0,0) -- (0,0) -- cycle;
\draw (0.5,-2.5) node{$2$};
\draw (0.5,-1.5) node{$\emphDes1$};
\end{tikzpicture}\\
t^{3\cdot 0+2} & t^{3\cdot 1+1}
\end{array}
\]
The signs $\epsilon_{\lambda,2}$ and $\epsilon_{\lambda,3}$ are positive, as all border strip tableaux in this example have even height.
For a primitive sixth root of unity we obtain a polynomial in which the signs of the nonzero coefficients do not coincide. Since the $6$-core of $\lambda=222$ is not empty, we do not have a combinatorial interpretation of $f^\lambda(\xi, t)$ in this case.
\end{example}

In the following sections we present the proof. The crucial steps are as follows: first, we use the Littlewood quotient map to bijectively map border strip tableaux to standard Young tableau tuples, and apply this to the right hand side of \cref{eq:mainThm}.
Next we express $f^\lambda(q,t)$ in terms of principal specialisations of the Schur function $s_\lambda$. For a primitive root of unity $\xi$, a generalisation of a theorem by V. Reiner, D. Stanton and D. White allows us to regard $f^\lambda(\xi,t)$ as generating function over the set of tuples of semistandard Young tableaux. Finally, a bijection links the two resulting expressions.

\section{The Littlewood quotient map}\label{sec:partitionQuotient}
We start this section by introducing the $k$-quotient of a partition $\lambda$ using a graphical description.

Interpret the lower-right contour of the Young diagram $\lambda$ as path with vertical and horizontal steps and append empty rows (vertical steps) to the bottom, such that the total number of rows is divisible by $k$.
Starting with the leftmost step in the lowest row, we label each step of the path in incremental order, starting with $0$.
For $0\le s\le k-1$ let $\lambda^s$ be the partition one obtains from the steps whose label is congruent to $s$ modulo $k$. 
The tuple $(\lambda^0,\lambda^1,\dots,\lambda^{k-1})$ is called the \defin{$k$-quotient} of $\lambda$.

\begin{example}
We construct the $3$-quotient of $\lambda=654222$.
For a nicer display we put the labels for the horizontal steps on top of each step and the labels for each vertical step to its right. As the number of rows is already divisible by $3$, we do not append empty rows.
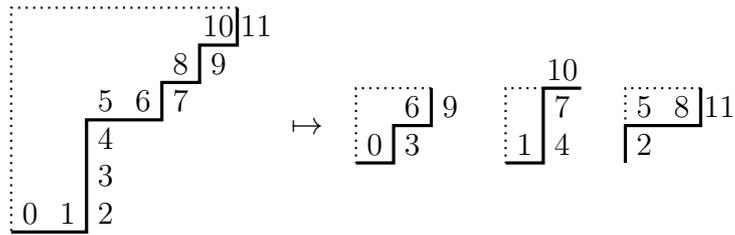
\[
\begin{tikzpicture}[x=1.2em, y=1.2em,baseline={([yshift=-1ex]current bounding box.center)}]
\draw[line width=1.2] (6,0) -- ++(0,-1) -- ++(-1,0) -- ++(0,-1) -- ++(-1,0) -- ++(0,-1) -- ++(-2,0) -- ++(0,-1) -- ++(-0,0) -- ++(0,-1) -- ++(-0,0) -- ++(0,-1) -- ++(-2,0);

\draw[line width=0.8,dotted] (6,0) -- (0,0) -- (0,-6);

\draw (0.5, -5.5) node{$0$} ++(1,0) node{$1$} ++(1,0) node{$2$} ++(0,1) node{$3$} ++(0,1) node{$4$} ++(0,1) node{$5$} ++(1,0) node{$6$} ++(1,0) node{$7$} ++(0,1) node{$8$} ++(1,0) node{$9$} ++(0,1) node{$10$} ++(1,0) node{$11$}; 
\end{tikzpicture}
\mapsto
\begin{array}{ccc}
\begin{tikzpicture}[x=1.2em, y=1.2em,baseline={([yshift=-1ex]current bounding box.center)}]
\draw[line width=1.2] (2,0) -- ++(0,-1) -- ++(-1,0) -- ++(0,-1) -- ++(-1,0);
\draw[line width=0.8,dotted] (2,0) -- (0,0) -- (0,-2);
\draw (0.5, -1.5) node{$0$} ++(1,0) node{$3$}  ++(0,1) node{$6$} ++(1,0) node{$9$}; 
\draw (0,1) node{}; \draw (0,-3) node{};
\end{tikzpicture}&
\begin{tikzpicture}[x=1.2em, y=1.2em,baseline={([yshift=-1ex]current bounding box.center)}]
\draw[line width=1.2] (2,0) -- ++(-1,0) -- ++(0,-2) -- ++(-1,0);
\draw[line width=0.8,dotted] (2,0) -- (0,0) -- (0,-2);
\draw (0.5, -1.5) node{$1$} ++(1,0) node{$4$}  ++(0,1) node{$7$} ++(0,1) node{$10$}; 
\draw (0,1) node{}; \draw (0,-3) node{};
\end{tikzpicture}&
\begin{tikzpicture}[x=1.2em, y=1.2em,baseline={([yshift=-1ex]current bounding box.center)}]
\draw[line width=1.2] (2,0) -- ++(0,-1) -- ++(-2,0) -- ++(0,-1);
\draw[line width=0.8,dotted] (2,0) -- (0,0) -- (0,-2);
\draw (0.5, -1.5) node{$2$} ++(0,1) node{$5$}  ++(1,0) node{$8$} ++(1,0) node{$11$};
\draw (0,1) node{}; \draw (0,-3) node{};
\end{tikzpicture}
\end{array}
\]
We obtain the $3$-quotient $(21,11,2)$.
\end{example}

Some fundamental properties of the quotient are the following.
\begin{proposition}[{G. James \& A. Kerber \cite[Chapter 2.7]{James1984}}]\label{thm:quotient}
\begin{enumerate}[(i)]
\item The function that maps a partition $\lambda\vdash n$ to its $k$-quotient is a bijection between the set of partitions with empty $k$-core and the set of $k$-tuples of partitions $(\lambda^0,\lambda^1,\dots,\lambda^{k-1})$ with $|\lambda^0|+|\lambda^1|+\dots+|\lambda^{k-1}| = \frac{n}{k}$.
\item Let $\lambda$ and $\mu$ be two partitions with empty $k$-core such that $\lambda/\mu$ is a border strip of size $k$.
Let $(\lambda^0,\dots, \lambda^{k-1})$ and $(\mu^0,\dots, \mu^{k-1})$ be the corresponding $k$-quotients for $\lambda$ and $\mu$, respectively.
Then there exists an index $0\le s\le k-1$ such that $|\lambda^s/\mu^s|=1$ and $\lambda^t = \mu^t$ for all $t\neq s$.

Alternatively, the $k$-quotient of $\lambda$ can be obtained from the $k$-quotient of $\mu$ by adding a single cell to one of the partitions.\label{thm:quotient_strip}
\end{enumerate}
\end{proposition}
In the following we denote multi-sets with two curly brackets, e.g. $\{\{a,a,b\}\}$, and for a multi-set $X$ and a real number $k$ we use the notions $X+k = \{\{x+k:x\in X\}\}$ and $kX= \{\{k\cdot x:x\in X\}\}$.
\begin{proposition}[I. Macdonald {\cite[Example I.1.8(d)]{Macdonald1995}}]\label{prop:hookvalues-quotient}
Denote with $h(\lambda)\coloneqq\{\{h(x): x\in\lambda\}\}$ the multi-set of the hook values of all cells of a partition $\lambda$ and denote with $c_{0,k}(\lambda)\coloneqq\{\{h(x): x\in\lambda, k\mid h(x)\}\}$ the multi-set of hook values, that are divisible by $k$.

Let $\lambda$ be a partition with empty $k$-core and $(\lambda^0,\dots, \lambda^{k-1})$ its $k$-quotient, then

\[
\frac{1}{k} h_{0,k}(\lambda) = \bigcup_{0\le i<k} h(\lambda^i).
\]
\end{proposition}

\begin{proposition}\label{prop:contents-quotient}
Denote with $c(\lambda)\coloneqq\{\{c(x): x\in\lambda\}\}$ the multi-set of the contents of all cells of a partition $\lambda$.
For $0\le i < k$ denote with $c_{r,k}(\lambda)\coloneqq\{\{c(x): x\in\lambda, k\mid (c(x)+r)\}\}$ the multi-set of contents, that are congruent to $-r$ modulo $k$.

Let $\lambda$ be a partition with empty $k$-core and $(\lambda^0,\dots, \lambda^{k-1})$ its $k$-quotient, then for $0\le r < k$

\[
\frac{1}{k}\left(c_{r,k}(\lambda)+r\right) = \bigcup_{0\le i<k-r} c(\lambda^i) \cup \bigcup_{k-r\le i<k}(c(\lambda^i)+1).
\]
\end{proposition}
\begin{proof}
For $r=0$ this result is known, see \cite[Example I.1.8(d)]{Macdonald1995}. Let $\lambda = (\lambda_1,\dots,\lambda_\ell)$ be the partition with as many zero parts appended, such that $k\mid \ell$ and let $\ell/k = \ell'$.

Let $\mu=(r^\ell)$ the partition consisting out of $\ell$ parts equal to $r$ and let $\lambda+\mu$ be the partition obtained from $\lambda$ and $\mu$ by pairwise addition of the parts.

We have
\begin{equation}\label{eq:shiftet_contents}
c_{r,k}(\lambda)+r = c_{0,k}(\mu+\lambda)\setminus c_{0,k}(\mu).
\end{equation}

Let further be $(\kappa^0,\dots,\kappa^{k-1})$ the $k$-quotient of $\lambda+\mu$, then
\begin{equation}\label{eq:shiftet_quotients}
\kappa^i = \begin{cases}
\lambda^{i-r}&\text{if $i\ge r$}\\
\lambda^{k+i-r}+(1^{\ell'})&\text{if $i<r$}
\end{cases}
\end{equation}

The general case can now be concluded from the case $r=0$ using equations \eqref{eq:shiftet_contents} and \eqref{eq:shiftet_quotients}.
\end{proof}

Let $\Lambda = (\lambda^0,\dots, \lambda^{k-1})$ be a tuple of Young diagrams and let $|\Lambda|$ be the total number of cells in $\Lambda$. A \defin{standard Young tableau tuple} of shapes $\Lambda$ is a bijective filling of the cells of $\Lambda$ with the values $\{1,\dots,|\Lambda|\}$ such that entries in each diagram strictly increase along rows from left to right and along columns from top to bottom. We denote the set of all such fillings with $\SYTT(\Lambda)$.

There exists a bijection between $\BST(\lambda, k)$ and $\SYTT(\Lambda)$, where $\Lambda$ is the $k$-quotient of $\lambda$.
Regarding a border strip tableau $B$ of shape $\lambda$ as a flag of partitions $\emptyset = \nu_0 \subset \nu_1 \subset \dots \subset \nu_{n/k} = \lambda$, we denote with $\Lambda_i$ the $k$-quotient of $\nu_i$. By \cref{thm:quotient} $(\emptyset,\dots,\emptyset)=\Lambda_0 \subset \Lambda_1 \subset \dots \subset \Lambda_{n/k}= \Lambda$ is a flag of tuples of Young diagrams, such that two consecutive tuples differ by precisely one cell. For $1\le i \le n/k$ denote with $x_i$ the unique cell in $\Lambda$ which is contained in $\Lambda_i$ but not $\Lambda_{i-1}$. By filling $x_i$ with $i$, we obtain a standard Young tableau tuple of shapes $\Lambda$.

This bijection is called \defin{Littlewood quotient map} by J. Haglund \cite{Haglund2007} or rim hook bijection by I. Pak \cite{Pak2000}. An example is given in \cref{fig:BSTtoSYTT}.

\begin{figure}
\centering

\[B = 
\begin{tikzpicture}[x=1em, y=1em,baseline={([yshift=-1ex]current bounding box.center)}]
\draw[line width=0.8] (0,0) -- (6,0) -- ++(0,-1) -- ++(-1,0) -- ++(0,-1) -- ++(-1,0) -- ++(0,-1) -- ++(-2,0) -- ++(0,-1) -- ++(-0,0) -- ++(0,-1) -- ++(-0,0) -- ++(0,-1) -- ++(-2,0) -- cycle;
\draw[line width=0.8] (0,0) -- (4,0) -- ++(0,-1) -- ++(-0,0) -- ++(0,-1) -- ++(-0,0) -- ++(0,-1) -- ++(-2,0) -- ++(0,-1) -- ++(-0,0) -- ++(0,-1) -- ++(-0,0) -- ++(0,-1) -- ++(-2,0);
\draw[line width=0.8] (0,0) -- (4,0) -- ++(0,-1) -- ++(-0,0) -- ++(0,-1) -- ++(-0,0) -- ++(0,-1) -- ++(-2,0) -- ++(0,-1) -- ++(-1,0) -- ++(0,-1) -- ++(-1,0);
\draw[line width=0.8] (0,0) -- (4,0) -- ++(0,-1) -- ++(-0,0) -- ++(0,-1) -- ++(-0,0) -- ++(0,-1) -- ++(-4,0);
\draw[line width=0.8] (0,0) -- (3,0) -- ++(0,-1) -- ++(-0,0) -- ++(0,-1) -- ++(-0,0) -- ++(0,-1) -- ++(-3,0);
\draw[line width=0.8] (0,0) -- (3,0) -- ++(0,-1) -- ++(-0,0) -- ++(0,-1) -- ++(-3,0);
\draw[line width=0.8] (0,0) -- (2,0) -- ++(0,-1) -- ++(-1,0) -- ++(0,-1) -- ++(-1,0);
\draw[line width=0.8] (0,0) -- (0,0);

\draw (4.5,-1.5) node{$7$};
\draw (0.5,-5.5) node{$6$};
\draw (0.5,-4.5) node{$\emphDes5$};
\draw (3.5,-2.5) node{$\emphDes4$};
\draw (0.5,-2.5) node{$3$};
\draw (1.5,-1.5) node{$\emphDes2$};
\draw (0.5,-1.5) node{$1$};
\end{tikzpicture}\quad \leftrightarrow \quad 
\begin{array}{lll}
\ytableaushort{3{\emphDes4},6}&\ytableaushort{1,{\emphDes5}}&\ytableaushort{{\emphDes2}7}
\end{array}=\mathcal{T}\]
\resizebox{\textwidth}{!} 
{$\begin{array}{cccccccccccccccc}
\text{\large $B$ as flag:}&\emptyset &\subset &21& \subset &33&  \subset& 333& \subset& 444& \subset& 44421& \subset& 444222 &\subset &654222\\
&\updownarrow&&\updownarrow&&\updownarrow&&\updownarrow&&\updownarrow&&\updownarrow&&\updownarrow&&\updownarrow
\\
\text{\large $\mathcal{T}$ as flag:}&(\emptyset,\emptyset,\emptyset) &\subset & (\emptyset,1,\emptyset)& \subset & (\emptyset,1,1) & \subset & (1,1,1) & \subset & (2,1,1)& \subset & (2,11,1)& \subset& (21,11,1)& \subset& (21,11,2)
\end{array}$}
\caption{The Littlewood quotient map applied to $B$. The descents of $B$ and $\mathcal{T}$ are highlighted. We have $\DES(B)=\DES(\mathcal{T})=\{2,4,5\}$.}
\label{fig:BSTtoSYTT}
\end{figure}
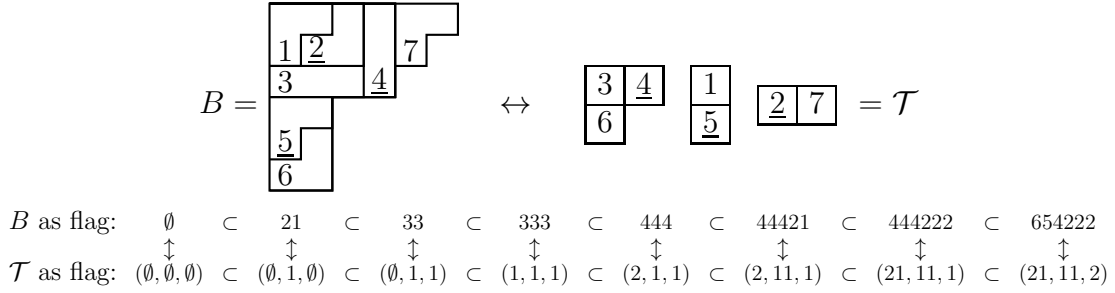

We now describe how descents are transported via this bijection.

\begin{definition}\label{def:desSYTT}
Let $\mathcal{T}=(T^0,\dots,T^{k-1})$ be a standard Young tableau tuple and let $c(x)$ be the content of an entry $x$ within its Young diagram.

Assume now that $i \in T^s$ and $(i+1) \in T^t$. 
Then $i$ is a descent of $\mathcal{T}$, if either $s\le t$ and $c(i)>c(i+1)$, or $s>t$ and $c(i)\ge c(i+1)$.
 $\DES(\mathcal T)$ denotes the set of all descents of $\mathcal T$.
\end{definition}

\begin{lemma}\label{lem:LittlewoodProperties}
Let $B\in \BST(\lambda,k)$ be a border strip tableau and let $\mathcal{T}=(T^0,\dots,T^{k-1})$ be the standard Young tableau tuple corresponding to $B$ via the Littlewood quotient map. Then $\DES(B) = \DES(\mathcal T)$.

Furthermore, let $s$ be the index of the unique Young diagram in $\mathcal T$ containing $1$ and let $B^1$ be the strip in $B$  containing $1$. Then $\height(B^1) = k-1-s\eqqcolon \idx_1(\mathcal T)$.
\end{lemma}
\begin{proof}
We first proof the following claim: For $i\in T^s$ let $c_{\mathcal{T}}(i)$ be the content of $i$ and let $c_B(i)$ the content of the tail of the strip containing $i$ in $B$, then $(c_B(i)-1) = k\cdot (c_{\mathcal{T}}(i)-1)+s$.

Recall that we obtain the $k$-quotient of a partition by first adding $0$ parts, such that the total number of rows $\ell$ is divisible by $k$ and then following the path of the lower-right contour of the Young diagram. We denote such a path with a finite binary sequence $w = (w_0,w_1,\dots)$ where a $0$ is a north step and a $1$ is an east step. Then the partitions in the $k$-quotient are given by the binary sequences $w^s \coloneqq (w_{i\cdot k+s})_{i\ge 0}$ for $0\le s < k$.

Let $\nu_i$ and $\nu_{i-1}$ be the partitions with $\ell$ (possibly zero) parts obtained from $B$ from all strips with labels at most $i$ and $i-1$ respectively.
Denote the path of $\nu_i$ with $(w_0,w_1,\dots)$ and the path of $\nu_{i-1}$ with $(w'_0,w'_1,\dots)$.
As $\nu_{i}/\nu_{i-1}$ is the border strip containing $i$, we have
\[
w_i = \begin{cases}1-w'_i=0 &\text{if $i=\ell+c_B(i)-1$,}\\
1-w'_i=1 &\text{if $i=\ell+c_B(i)-1+k$,}\\
w'_i &\text{else}\\
\end{cases}.
\]
Let $c_B(i)-1 = k\cdot q+r$ with integers $q$ and $0\le r< k$ and let $\ell = k\cdot \hat{\ell}$. Thus $w_{\ell+c_B(i)-1} = w_{k\cdot (\hat{\ell}+q)+r}$ is the $(\hat{\ell}+q)$-th step in the path of the $r$-th partition in the $k$-quotient of $\nu_i$. Similarly $w_{\ell+c_B(i)-1+k}$ is the $(\hat{\ell}+q+1)$-th step in the path of the same partition. The same holds for $w'_{\ell+c_B(i)-1}$ and $w'_{\ell+c_B(i)-1+k}$ with respect to the $r$-th partition of the $k$-quotient of $v_{i-1}$.
Therefore $r=s$ and $q = c_{\mathcal{T}}(i)-1$, which proves the claim.

Observe that $i$ is a descent in $B$, if and only if $c_B(i) > c_B(i+1)$. Moreover, let $i \in T^s$ and $(i+1) \in T^t$, then $i$ is a descent in $\mathcal{T}$ if and only if $k\cdot c_{\mathcal{T}}(i)+s>k\cdot c_{\mathcal{T}}(i+1)+t$. As $(c_B(i)-1) = k\cdot (c_{\mathcal{T}}(i)-1)+s$, we obtain $\DES(B) = \DES(\mathcal T)$.

For the second part of the theorem, note that
\[c_B(1)-1 = -1-\height(B^1) \equiv k-1-\height(B^1) \pmod{k}.\]
Thus $\height(B^1) = k-1-s$.
\end{proof}

The following graphical description may be helpful for understanding \cref{def:desSYTT}: Tilt the Young diagrams in $\mathcal{T}$ by $45$ degrees in counter-clockwise direction and align them such that cells with the same content lay on a horizontal line. We call this the \defin{"Austrian notation"}.
Then $i$ is a descent in $\mathcal{T}$ if and only if $(i+1)$ is
\begin{itemize}
\item in a tableau to the left and weakly below $i$, or
\item in the same tableau or a tableau to the right and strictly below $i$.
\end{itemize}
An example of the graphical description is given in \cref{fig:austrianNotation}.

\begin{remark}
\cref{def:desSYTT} agrees with the definition of descents using the content reading order in the Bylund--Haiman model \cite[Equations (77) and (80)]{Haglund2005}, which is used in the quasisymmetric expansion of LLT-polynomials. B. Westbury made us aware of this relation when we first presented our results to a broader audience.
\end{remark}

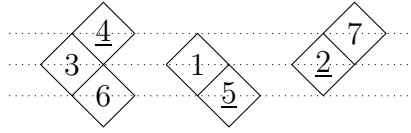
\begin{figure}

\[
\begin{tikzpicture}[x=1em,y=1em]
\draw (1,1) -- ++(2,-2) -- ++(-1,-1) -- (0,0) -- ++(2,2) -- ++(1,-1) -- ++(-2,-2);
\draw (1,0) node{$3$};
\draw (2,1) node{$\emphDes4$};
\draw (2,-1) node{$6$};

\draw (4,0) -- ++(1,1) -- ++(2,-2) -- ++(-1,-1) -- ++(-2,2)  (6,0) -- ++(-1,-1);
\draw (5,0) node{$1$};
\draw (6,-1) node{$\emphDes5$};

\draw (8,0) -- ++(2,2) -- ++(1,-1) -- ++(-2,-2) -- ++(-1,1)  (9,1) -- ++(1,-1);
\draw (9,0) node{$\emphDes2$};
\draw (10,1) node{$7$};

\draw[dotted] (-1,0) -- ++(13,0);
\draw[dotted] (-1,1) -- ++(13,0);
\draw[dotted] (-1,-1) -- ++(13,0);

\node at (12,1.5) (c22) {};

\end{tikzpicture}
\]
\caption{The standard Young tableau tuple in "Austrian notation". On each dotted horizontal line are cells with the same content.}\label{fig:austrianNotation}
\end{figure}

We conclude this section by applying the Littlewood quotient map to the right hand side of \cref{eq:mainThm}.

\begin{lemma}\label{thm:SYT-tuples}
Let $\lambda$ be a partition with empty $k$-core and let $\Lambda = (\lambda^0,\dots, \lambda^{k-1})$ be its $k$-quotient,
then
\begin{equation}\label{eq:SYT-tuples}
\sum_{B\in \BST(\lambda,k)}t^{ k\cdot |\DES(B)| + \height(B^1)} = \sum_{\mathcal{T}\in \SYTT(\Lambda)}t^{ k\cdot |\DES(\mathcal{T})| + \idx_1(\mathcal{T})}.
\end{equation}
\end{lemma}
\begin{proof}
This is a direct consequence of \cref{lem:LittlewoodProperties}.
\end{proof}

\section{Schur functions}\label{sec:schur}
In this section we express $f^\lambda(q,t)$ in terms of Schur functions and evaluate it at roots of unity.
\begin{definition}
A \defin{semistandard Young tableau} of shape $\lambda$ is a filling of the Young diagram with positive integers such that
the rows are weakly increasing from left to right and
the columns are strictly increasing from top to bottom. Let \defin{$\SSYT(\lambda)$} be the set of all semistandard Young tableaux of shape $\lambda$.

For a semistandard Young tableau $T$ we associate the monomial $\xvec^T = \prod_{i\ge 1}x_i^{t_i}$ where $t_i$ is the number of occurrences of the number $i$ in $T$. The \defin{Schur function} $s_\lambda$ associated with $\lambda$ is the generating function
\[
s_{\lambda}(x_1,x_2,x_3,\dots) = \sum_{T\in \SSYT(\lambda)}\xvec^T.
\]
\end{definition}
When specialising precisely $m$ variables with $1$ and the rest with zero, one obtains the number of semistandard Young tableaux of shape $\lambda$ with entries bounded from above by $m$. We also denote this with
\[s_\lambda(\underbrace{1,\dots,1}_{m},0,0,\dots) = s_\lambda(1^m)=\#\text{$\SSYT$ of shape $\lambda$ with all entries $\le m$}.\]

The following identity relates $f^\lambda(q,t)$ to the principal specialisations of the the Schur function $s_\lambda$.

\begin{theorem}[R. Stanley {\cite[Proposition 8.3]{Stanley1971}}]
\label{thm:Stanley}
Let $\lambda\vdash n$ and let $(t;q)_{n+1} = (1-t)(1-tq)\dots(1-tq^n)$ be the $q$-Pochhammer-symbol, then
\[
\frac{f^{\lambda}(q,t)}{(t;q)_{n+1}} = \sum_{m=0}^\infty t^m s_\lambda (1,q,\dots,q^m).
\]
\end{theorem}

This identity is particularly useful for us because the $q$-Pochhammer-symbol, as well as the principal specialisations of the Schur functions, can be easily evaluated at roots of unity.

\begin{proposition}\label{prop:PochhammerEvaluation}
Given $k\mid n$ and a primitive $k$-th root of unity $\xi$, we have that \[(t;\xi)_{n+1} = (1-t)(1-t^k)^{n/k}.\]
\end{proposition}
\begin{proof}
Because $\xi$ is a $k$-th primitive root of unity we have $\xi^{\ell k+r}=\xi^r$ for integers $\ell$ and $r$ and $x^k-1 = \prod_{i=0}^{k-1}(x-\xi^i)$. We obtain:
\begin{align*}
(t;\xi)_{n+1} &= (1-t)(1-t\xi)\dots(1-t\xi^n) = (1-t)\left(\prod_{i=0}^{k-1}(1-t\xi^i)\right)^{n/k}\\
&= (1-t)\left(t^k\prod_{i=0}^{k-1}(1/t-\xi^i)\right)^{n/k} = (1-t)\left(t^k(1/t^k-1)\right)^{n/k}\\
&= (1-t)(1-t^k)^{n/k}.
\end{align*}
\end{proof}

\begin{theorem}\label{thm:SchurEvaluation}
Let $\lambda\vdash n$ be a partition with empty $k$-core and with $k$-quotient $(\lambda^0,\dots,\lambda^{k-1})$, and let $\xi$ be a primitive $k$-th root of unity.
If $m = \ell\cdot k+r$ for $0\le r < k$, then
\begin{equation}\label{eq:SchurEvaluation}
s_\lambda(1,\xi,\dots, \xi^{m-1}) =\epsilon_{\lambda,k}\cdot s_{\lambda^0}(1^\ell)\cdots s_{\lambda^{k-r-1}}(1^\ell)\cdot s_{\lambda^{k-r}}(1^{\ell+1}) \cdots s_{\lambda^{k-1}}(1^{\ell+1}),
\end{equation}
where $\epsilon_{\lambda,k}$ is the sign from \cref{thm:sign}.
\end{theorem}
\begin{remark} The case $r=0$ in \cref{thm:SchurEvaluation} is a theorem due to V. Reiner, D. Stanton and D. White \cite[Theorem 4.3]{Reiner2004}.
\end{remark}

\begin{proof}
Let $b(\lambda)=\sum(i-1)\lambda_i$ and $R_{\lambda,k}=\xi^{b(\lambda)}\frac{\prod_{x\in \lambda, k\nmid m+c(x)}(1-\xi^{m+c(x)})}{\prod_{x\in \lambda, k\nmid h(x)}(1-\xi^{h(x)})}$, then by \cref{prop:hookvalues-quotient}, \cref{prop:contents-quotient} and the hook-content formula~\cite[Theorem 7.21.2]{Stanley1999} we get:

\begin{align*}s_\lambda(1,\xi,\dots&, \xi^{m-1}) = \xi^{b(\lambda)}\prod_{x\in \lambda}\frac{1-\xi^{m+c(x)}}{1-\xi^{h(x)}}=R_{\lambda,k} \cdot \frac{\prod_{x\in \lambda, k\mid m+c(x)}(1-\xi^{m+c(x)})}{\prod_{x\in \lambda, k\mid h(x)}(1-\xi^{h(x)})}\\
&= R_{\lambda, k}\cdot \frac{\prod_{x\in \lambda, k\mid m+c(x)}(m+c(x))}{\prod_{x\in \lambda, k\mid h(x)}h(x)}= R_{\lambda,k}\cdot \frac{\prod_{x\in \lambda, k\mid r+c(x)}(\ell +(r+c(x))/k)}{\prod_{x\in \lambda, k\mid h(x)}h(x)/k}\\
&= R_{\lambda,k}\cdot \frac{\prod_{u \in \frac{1}{k}(c_{r,k}(\lambda)+r)}(\ell +u)}{\prod_{u \in \frac{1}{k}h_{0,k}(\lambda)}u}\\
&= R_{\lambda,k}\cdot \frac{\prod_{u \in \bigcup_{0\le i<k-r} c(\lambda^i) \cup \bigcup_{k-r\le i<k}(c(\lambda^i)+1)}(\ell +u)}{\prod_{u \in \bigcup_{0\le i<k} h(\lambda^i)}u}\\
&= R_{\lambda,k}\cdot \frac{\prod_{u \in \bigcup_{0\le i<k-r} c(\lambda^i)}(\ell +u)\cdot \prod_{u \in \bigcup_{k-r\le i<k}c(\lambda^i)}(\ell+1 +u)}{\prod_{u \in \bigcup_{0\le i<k-r} h(\lambda^i)}u \cdot \prod_{u \in \bigcup_{k-r\le i<k} h(\lambda^i)}u}\\
&= R_{\lambda,k} \prod_{0\le i < k-r} s_{\lambda^i}(1^\ell)\cdot \prod_{k-r\le i < k}s_{\lambda^i}(1^{\ell+1}).
\end{align*}

As $\lambda$ has empty $k$-core each content modulo $k$ appears $n/k$ times. Hence,
\[R_{\lambda,k}=\xi^{b(\lambda)}\frac{\prod_{x\in \lambda, k\nmid m+c(x)}(1-\xi^{m+c(x)})}{\prod_{x\in \lambda, k\nmid h(x)}(1-\xi^{h(x)})}=\xi^{b(\lambda)}\frac{\prod_{0<i\le k-1}(1-\xi^{i})^{n/k}}{\prod_{x\in \lambda, k\nmid h(x)}(1-\xi^{h(x)})},
\]
which does not depend on $m$.

As the formula \eqref{eq:SchurEvaluation} is satisfied for $r=0$ due to  \cite[Theorem 4.3]{Reiner2004} by V. Reiner, D. Stanton and D. White, we obtain $R_{\lambda,k}= \epsilon_{\lambda,k}$.
\end{proof}

The right hand side of \cref{eq:SchurEvaluation} can, up to sign, be interpreted combinatorially as the number of tuples of semistandard Young diagrams $(T^0,T^1,\dots, T^{k-1})$ such that
\begin{itemize}
\item the tableau $T^i$ has shape $\lambda^i$,
\item the tableaux $T^0,T^1,\dots,T^{k-r-1}$ have entries between $2$ and $\ell+1$, and
\item the tableaux  $T^{k-r},T^{k-r+1},\dots,T^{k-1}$ have entries between $1$ and $\ell+1$.
\end{itemize}
Keeping the notation of \cref{thm:SchurEvaluation}, this yields:

\begin{lemma}\label{thm:SSYT-tuples}
Denote with $\SSYTT(\Lambda)$ the set of all tuples of semistandard Young tableaux with shapes $\Lambda = (\lambda^0,\dots,\lambda^{k-1})$ that contain at least one $1$.
For such a tuple $\mathfrak{T}$, let $\max(\mathfrak{T})$ be the maximal entry. Let $s$ be the index of the leftmost tableau containing $1$ and set $\idx_1(\mathfrak{T})\coloneqq k-1-s$. Then
\begin{align}\label{eq:SSYT-tuples}
\frac{1}{(1-t^k)^{n/k-1}} f^{\lambda}(\xi,t) = & (1-t)(1-t^k) \sum_{m=0}^\infty t^m s_\lambda (1,\xi,\dots,\xi^m) =\nonumber \\
&\epsilon_{\lambda,k}\cdot\sum_{\mathfrak{T} \in \SSYTT(\Lambda)} t^{k\cdot (\max(\mathfrak{T})-1) + \idx_1(\mathfrak{T})}.
\end{align}

\end{lemma}
\begin{proof}
The first equation is a direct consequence of \cref{thm:Stanley} and \cref{prop:PochhammerEvaluation}.

By the combinatorial interpretation of the right hand side of \cref{eq:SchurEvaluation} we obtain for $m=k\cdot\ell+r$
\begin{align*}
\epsilon_{\lambda,k}(s_\lambda (1,\xi,\dots,\xi^m) &- s_\lambda (1,\xi,\dots,\xi^{m-1})) =\\
&|\{\mathfrak{T} \in \SSYTT(\Lambda): \idx_1(\mathfrak{T})=r, \max(\mathfrak{T})\le \ell+1\}|
\end{align*}
and for $m=0=k\cdot 0+0$
\[
\epsilon_{\lambda,k}s_\lambda (1) =
|\{\mathfrak{T} \in \SSYTT(\Lambda): \idx_1(\mathfrak{T})=0, \max(\mathfrak{T})= 1\}|.
\]
Thus
\begin{align*}
\epsilon_{\lambda,k}&\cdot(1-t)(1-t^k) \sum_{m=0}^\infty t^m s_\lambda (1,\xi,\dots,\xi^m) = \\
&\epsilon_{\lambda,k}(1-t^k)\left(t^0s_{\lambda}(1)+\sum_{m=1}^\infty t^m(s_\lambda (1,\xi,\dots,\xi^m) - s_\lambda (1,\xi,\dots,\xi^{m-1}))\right)=\\
&(1-t^k)\sum_{\ell=0}^{\infty}\sum_{r=0}^{k-1}t^{k\cdot\ell+r}|\{\mathfrak{T} \in \SSYTT(\Lambda): \idx_1(\mathfrak{T})=r, \max(\mathfrak{T})\le \ell+1\}|=\\
&\sum_{\ell=0}^{\infty}\sum_{r=0}^{k-1}t^{k\cdot\ell+r}|\{\mathfrak{T} \in \SSYTT(\Lambda): \idx_1(\mathfrak{T})=r, \max(\mathfrak{T})= \ell+1\}|=\\
&\sum_{\mathfrak{T} \in \SSYTT(\Lambda)} t^{k\cdot (\max(\mathfrak{T})-1) + \idx_1(\mathfrak{T})}.
\end{align*}
\end{proof}

\section{The final bijection}\label{sec:bijection}
In this section we discus the bijection that links together \cref{eq:SYT-tuples} and \cref{eq:SSYT-tuples} and proves our main result \cref{thm:mainThm}.

Denote with $C_p$ the set of weak compositions with precisely $p$ parts, that is the set of $p$-tuples of non negative integers. Let $\Lambda = (\lambda^0,\dots,\lambda^{k-1})$ be a $k$-tuple of partitions and let $\ell = |\lambda^0|+\dots+|\lambda^{k-1}|$.
We now present a bijection
\[
\phi: C_{\ell-1} \times \SYTT(\Lambda) \to \SSYTT(\Lambda).
\]
Fix $\alpha = (\alpha_1,\dots,\alpha_{\ell-1})\in C_{\ell-1}$ and $\mathcal{T}\in \SYTT(\Lambda)$. For $1\le s \le \ell$, let $x_s$ be the unique cell in $\mathcal{T}$ that contains $s$ and let $d_s$ be the number of descents in $\mathcal{T}$ that are strictly smaller than $s$.
Let $\mathfrak{T}$ be the tuple of semistandard Young tableaux obtained by filling the cell $x_s$ with $1+d_s+\sum_{i=1}^{s-1}\alpha_i$ and set $\phi(\alpha,\mathcal{T})=\mathfrak{T}$.
An example is given in \cref{fig:bijection}.

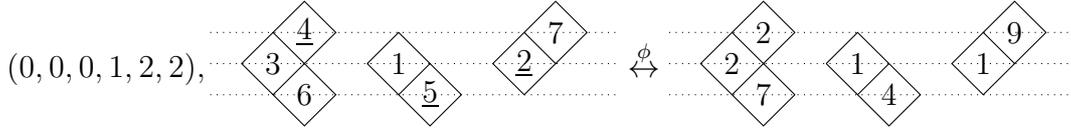
\begin{figure}

\[(0,0,0,1,2,2),
\begin{tikzpicture}[x=1em,y=1em,baseline={([yshift=-1ex]current bounding box.center)}]
\draw (1,1) -- ++(2,-2) -- ++(-1,-1) -- (0,0) -- ++(2,2) -- ++(1,-1) -- ++(-2,-2);
\draw (1,0) node{$3$};
\draw (2,1) node{$\emphDes{4}$};
\draw (2,-1) node{$6$};

\draw (4,0) -- ++(1,1) -- ++(2,-2) -- ++(-1,-1) -- ++(-2,2)  (6,0) -- ++(-1,-1);
\draw (5,0) node{$1$};
\draw (6,-1) node{$\emphDes{5}$};

\draw (8,0) -- ++(2,2) -- ++(1,-1) -- ++(-2,-2) -- ++(-1,1)  (9,1) -- ++(1,-1);
\draw (9,0) node{$\emphDes{2}$};
\draw (10,1) node{$7$};

\draw[dotted] (-1,0) -- ++(13,0);
\draw[dotted] (-1,1) -- ++(13,0);
\draw[dotted] (-1,-1) -- ++(13,0);
\end{tikzpicture}
\stackrel{\phi}{\leftrightarrow} \begin{tikzpicture}[x=1em,y=1em,baseline={([yshift=-1ex]current bounding box.center)}]
\draw (1,1) -- ++(2,-2) -- ++(-1,-1) -- (0,0) -- ++(2,2) -- ++(1,-1) -- ++(-2,-2);
\draw (1,0) node{$2$};
\draw (2,1) node{$2$};
\draw (2,-1) node{$7$};

\draw (4,0) -- ++(1,1) -- ++(2,-2) -- ++(-1,-1) -- ++(-2,2)  (6,0) -- ++(-1,-1);
\draw (5,0) node{$1$};
\draw (6,-1) node{$4$};

\draw (8,0) -- ++(2,2) -- ++(1,-1) -- ++(-2,-2) -- ++(-1,1)  (9,1) -- ++(1,-1);
\draw (9,0) node{$1$};
\draw (10,1) node{$9$};

\draw[dotted] (-1,0) -- ++(13,0);
\draw[dotted] (-1,1) -- ++(13,0);
\draw[dotted] (-1,-1) -- ++(13,0);
\end{tikzpicture}
\]
\caption{Bijectively mapping a composition and a standard Young tableau tuple to a tuple of semistandard Young tableaux.}\label{fig:bijection}
\end{figure}
\begin{remark}
This map fits into the setting of the MacMahon Verfahren (see \cite{MacMahon1913} and \cite[Chapter 10]{Lothaire2002}) and the theory of $P$-partitions developed by R. Stanley \cite{Stanley1971}.
\end{remark}

\begin{proposition}
The function $\phi: C_{\ell-1} \times \SYTT(\Lambda) \to \SSYTT(\Lambda)$ is a bijection and for $\phi(\alpha,\mathcal{T})=\mathfrak{T}$, we have
\[
|\alpha|+|\DES(\mathcal T)|+1 = \max(\mathfrak{T}) \quad\text{and}\quad \idx_1(\mathcal{T})=\idx_1(\mathfrak{T}),
\]
where $|\alpha|=\alpha_1+\dots+\alpha_{\ell-1}$ is the sum of its parts.
\end{proposition}
\begin{proof}
By construction the leftmost one in $\mathfrak{T}$ coincides with the one in $\mathcal{T}$, so $\idx_1(\mathcal{T})=\idx_1(\mathfrak{T})$. The maximal entry of $\mathfrak{T}$ is the filling of $x_{\ell}$, which is $1+d_{\ell}+\sum_{i=1}^{\ell-1}\alpha_i = |\alpha|+|\DES(\mathcal T)|+1$.

It remains to show, that this map is a bijection, by giving the inverse map.
To obtain $\mathcal{T}$ follow the entries from $\mathfrak{T}$ in increasing order and fill the corresponding cells of $\mathcal{T}$ with the smallest positive integer, which is not already used. If some values in $\mathfrak{T}$ coincide fill the corresponding cells of $\mathcal{T}$ in the unique order, that does not create any descents.
Having now $\mathfrak{T}$ and $\mathcal{T}$ we can again calculate the values $d_s$ and obtain the composition $\alpha$ from the fact that the cell $x_s$ is filled with $1+d_s+\sum_{i=1}^{s-1}\alpha_i$ in $\mathfrak{T}$.
\end{proof}
From this proposition we get:

\begin{lemma}\label{thm:bijection}
Let $\lambda\vdash n$ be a partition with empty $k$-core and let $\Lambda = (\lambda^0,\dots,\lambda^{k-1})$ its $k$-quotient. Then
\[
\frac{1}{(1-t^k)^{n/k-1}} \sum_{\mathcal{T}\in \SYTT(\Lambda)}t^{ k\cdot |\DES(\mathcal{T})| + \idx_1(\mathcal{T})} = \sum_{\mathfrak{T} \in \SSYTT(\Lambda)} t^{k\cdot (\max(\mathfrak{T})-1) + \idx_1(\mathfrak{T})}.
\]
\end{lemma}
\begin{proof}
By our bijection $\phi$ we have in terms of generating functions
\[
\frac{1}{(1-x)^{n/k-1}} \sum_{\mathcal{T}\in \SYTT(\Lambda)}x^{|\DES(\mathcal{T})|}y^{\idx_1(\mathcal{T})} = \sum_{\mathfrak{T} \in \SSYTT(\Lambda)} x^{\max(\mathfrak{T})-1} y^{\idx_1(\mathfrak{T})}.
\]
The theorem follows by substituting $x=t^k$ and $y=t$.
\end{proof}

Now we can conclude our main theorem.

\begin{proof}[Proof of \cref{thm:mainThm}]
Combining \cref{thm:SYT-tuples}, \cref{thm:SSYT-tuples} and \cref{thm:bijection} yields the desired identity.
\end{proof}

\smallskip
\noindent \textbf{Acknowledgements.}
The author would like to thank Martin Rubey, Esther Banaian, Bruce Westbury and Christian Krattenthaler for helpful insights and discussions related to this work.
Moreover the author would like to thank the organisers of the Graduate Online Combinatorics Colloquium (GOCC) Galen Dorpalen-Barry, Alex McDonough and Andr\'es R. Vindas Mel\'endez for the opportunity to present this results and gain feedback at an early stage.

\printbibliography

\end{document}